\documentclass[a4paper,USenglish,cleveref,autoref,thm-restate]{lipics-v2021}
\hideLIPIcs\def\subjclassHeading{}\ccsdesc{}\global\renewcommand\ccsdesc[2][100]{}
\nolinenumbers
\title{On a Problem of Steinhaus}

\author{Marcin Anholcer}
{Institute of Informatics and Electronic Economics, Pozna\'n University of Economics and Business, Pozna\'n, Poland}
{m.anholcer@ue.poznan.pl}
{https://orcid.org/0000-0001-7322-7095}
{Partially supported by the National Science Center of Poland under grant no.\ 2020/37/B/ST1/03298.}

\author{Bart{\l}omiej Bosek}
{Institute of Theoretical Computer Science, Faculty of Mathematics and Computer Science, Jagiellonian University, Krak{\'o}w, Poland}
{bartlomiej.bosek@uj.edu.pl}
{https://orcid.org/0000-0001-8756-3663}
{Partially supported by the National Science Center of Poland under grant no.\ 2019/35/B/ST6/02472.}

\author{Jaros{\l}aw Grytczuk}
{Faculty of Mathematics and Information Science, Warsaw University of Technology, Warsaw, Poland}
{j.grytczuk@mini.pw.edu.pl}
{https://orcid.org/0000-0002-0258-6143}
{Partially supported by the National Science Center of Poland under grant no.\ 2020/37/B/ST1/03298.}

\author{Grzegorz Gutowski}
{Institute of Theoretical Computer Science, Faculty of Mathematics and Computer Science, Jagiellonian University, Krak{\'o}w, Poland}
{grzegorz.gutowski@uj.edu.pl}
{https://orcid.org/0000-0003-3313-1237}
{Partially supported by the National Science Center of Poland under grant no.\ 2019/35/B/ST6/02472.}

\author{Jakub Przyby{\l}o}
{AGH University of Science and Technology, Faculty of Applied Mathematics, al.~A.~Mickiewicza~30, 30-059 Krak{\'o}w, Poland}
{jakubprz@agh.edu.pl}
{https://orcid.org/0000-0002-1262-7017}
{}

\author{Rafa{\l} Pyzik}
{Institute of Theoretical Computer Science, Faculty of Mathematics and Computer Science, Jagiellonian University, Krak{\'o}w, Poland}
{rafalpyzik@gmail.com}
{https://orcid.org/0000-0003-4147-7000}
{}

\author{Mariusz Zaj\k{a}c}
{Faculty of Mathematics and Information Science, Warsaw University of Technology, Warsaw, Poland}
{m.zajac@mini.pw.edu.pl}
{https://orcid.org/0000-0002-2080-9523}
{Partially supported by the National Science Center of Poland under grant no.\ 2019/35/B/ST6/02472.}

\authorrunning{M. Anholcer, B. Bosek, J. Grytczuk, G. Gutowski, J. Przyby{\l}o, R. Pyzik, M. Zaj\k{a}c}
\Copyright{Marcin Anholcer, Bart{\l}omiej Bosek, Jaros{\l}aw Grytczuk, Grzegorz Gutowski, Jakub Przyby{\l}o, Rafa{\l} Pyzik, Mariusz Zaj\k{a}c}

\ccsdesc[500]{Mathematics of computing $\rightarrow$ Discrete mathematics}

\keywords{17 Points Problem, Irregularities of Distribution, Farey Sequences, Stick-breaking Sequences}

\usepackage{tikz}
\usetikzlibrary{arrows,math,patterns,shapes}
\usepackage{todonotes}

\let\leq\leqslant
\let\geq\geqslant

\let\rho\varrho

\newcommand{\brac}[1]{{\left(#1\right)}}
\newcommand{\lbrac}[1]{{\left[#1\right)}}

\newcommand{\sbrac}[1]{{\left[#1\right]}}

\newcommand{\set}[1]{\left\{#1\right\}}
\newcommand{\seq}[1]{\left(#1\right)}
\newcommand{\norm}[1]{{\left|#1\right|}}
\newcommand{\floor}[1]{{\left\lfloor #1 \right\rfloor}}
\newcommand{\ceil}[1]{{\left\lceil #1 \right\rceil}}

\newcommand{\Oh}[1]{O\brac{#1}}
\newcommand{\oh}[1]{o\brac{#1}}

\newcommand{\Real}{\mathbb{R}}
\newcommand{\PosReal}{\Real_{+}}
\newcommand{\Rational}{\mathbb{Q}}

\DeclareMathOperator{\FI}{FI}
\DeclareMathOperator{\AI}{AI}
\DeclareMathOperator{\FP}{FP}

\DeclareMathOperator{\modulo}{mod}
\DeclareMathOperator{\sort}{sorted}

\newcommand{\ie}{\textit{i}.\textit{e}., }

\begin{document}

\maketitle

\begin{abstract}
Let $N$ be a positive integer.
A sequence $X=(x_1,x_2,\ldots,x_N)$ of points in the unit interval $[0,1)$ is \emph{piercing} if $\{x_1,x_2,\ldots,x_n\}\cap \left[\frac{i}{n},\frac{i+1}{n} \right) \neq\emptyset$ holds for every $n=1,2,\ldots, N$ and every $i=0,1,\ldots,n-1$.
In 1958 Steinhaus asked whether piercing sequences can be arbitrarily long.
A negative answer was provided by Schinzel, who proved that any such sequence may have at most $74$ elements.
This was later improved to the best possible value of $17$ by Warmus, and independently by Berlekamp and Graham.

In this paper we study a more general variant of piercing sequences.
Let $f(n)\geq n$ be an infinite nondecreasing sequence of positive integers.
A sequence $X=(x_1,x_2,\ldots,x_{f(N)})$ is \emph{$f$-piercing} if $\{x_1,x_2,\ldots,x_{f(n)}\}\cap \left[\frac{i}{n},\frac{i+1}{n} \right) \neq\emptyset$ holds for every $n=1,2,\ldots, N$ and every $i=0,1,\ldots,n-1$.
A special case of $f(n)=n+d$, with $d$ a fixed nonnegative integer, was studied by Berlekamp and Graham.
They noticed that for each $d\geq 0$, the maximum length of any $(n+d)$-piercing sequence is finite.
Expressing this maximum length as $s(d)+d$, they obtained an exponential upper bound on the function $s(d)$, which was later improved to $s(d)=O(d^3)$ by Graham and Levy.
Recently, Konyagin proved that $2d\leqslant s(d)< 200d$ holds for all sufficiently big $d$.

Using a different technique based on the Farey fractions and stick-breaking games, we prove here that the function $s(d)$ satisfies $\floor{c_1d}\leqslant s(d)\leqslant c_2d+o(d)$ where $c_1=\frac{\ln 2}{1-\ln 2}\approx2.25$ and $c_2=\frac{1+\ln2}{1-\ln2}\approx5.52$.
We also prove that there exists an infinite $f$-piercing sequence with $f(n)= \gamma n+o(n)$ if and only if $\gamma\geq\frac{1}{\ln 2}\approx 1.44$.

\end{abstract}

\section{Introduction}
\label{sec:intro}

In his book \emph{Sto zada\'{n}}~\cite{Steinhaus58} from 1958, in English translation \emph{One Hundred Problems in Elementary Mathematics}~\cite{Steinhaus64}, Steinhaus posed the following problem (Problems~6 and~7 in Chapter~1):

Does there exist for every positive integer $N$ a sequence $X=(x_1,x_2,\ldots,x_N)$ of real numbers in $[0,1)$ such that $\{x_1,x_2,\ldots,x_n\}\cap \left[\frac{i}{n},\frac{i+1}{n} \right) \neq\emptyset$ holds for every $n=1,2,\ldots, N$ and every $i=0,1,\ldots,n-1$?

We call such sequences \emph{piercing} of \emph{order} $N$.
The first solution was given by Schinzel (see~\cite{Steinhaus64}) who proved that no piercing sequence can have more than $74$ elements.
Then, Warmus~\cite{Warmus76} gave a complete solution by proving that the longest piercing sequence has $17$ elements and constructing essentially all the possible solutions of which there are $768$.

In the present paper we study a more general variant of piercing sequences. Let $f(n)\geq n$ be an infinite nondecreasing sequence of positive integers.
A sequence $X=(x_1,x_2,\ldots,x_{f(N)})$ is called \emph{$f$-piercing} of \emph{order} $N$ if $\{x_1,x_2,\ldots,x_{f(n)}\}\cap \left[\frac{i}{n},\frac{i+1}{n} \right) \neq\emptyset$ holds for every $n=1,2,\ldots, N$ and every $i=0,1,\ldots,n-1$.
The original question of Steinhaus concerns the extremal case of $f(n)=n$.

A natural class of $f$-piercing sequences with the function $f(n)=n+d$, where $d$ is a fixed nonnegative integer, was introduced by Berlekamp and Graham~\cite{BerlekampG70}.
They noticed that a fundamental result of Roth~\cite{Roth54} on the discrepancy of sequences implies easily that for each $d\geq 0$, the maximum order of any $(n+d)$-piercing sequence is bounded.
Denoting this order by $s(d)$, they proved (independently of Warmus) that $s(0)=17$ and gave an exponential upper bound on $s(d)$.
Later, a proof of the polynomial upper bound $s(d)\leqslant O(d^3)$ was sketched by Graham~\cite{Graham13}.
A full proof of this bound was recently completed by Levy~\cite{Levy20}, who also provided a lower bound $s(d)\geqslant\Omega(\sqrt{d})$.
These results were very recently improved by Konyagin~\cite{Konyagin21}, who proved that $s(d)\geqslant 2d$ for all $d\geq 0$, and $s(d)<200d$ for all $d\geqslant 4\cdot10^{16}$.

Using different techniques, we obtain here a further improvement on the bounds for the function $s(d)$.
One of our main results reads as follows.

\begin{restatable}{theorem}{thms}\label{thm:s}
The function $s(d)$ satisfies $\floor{c_1d}\leqslant s(d)\leqslant c_2d+o(d)$ where $c_1=\frac{\ln 2}{1-\ln 2}\approx2.25$ and $c_2=\frac{1+\ln2}{1-\ln2}\approx5.52$.
\end{restatable}

A natural question arising from the problem of Steinhaus is to determine the minimum growth of a function $f(n)$ allowing for arbitrarily long $f$-piercing sequences.
It is not hard to see that $f(n)=2n$ is one such function.
Actually, it allows even for an \emph{infinite} $f$-piercing sequence $X=(x_1,x_2,\ldots)$, defined by the natural condition that every prefix of $X$ of the form $(x_1,x_2,\ldots,x_{f(N)})$ is an $f$-piercing sequence of order $N$, for every $N\geqslant 1$.
We prove the following general result.

\begin{restatable}{theorem}{thmgood}\label{thm:good}
There exists an infinite $f$-piercing sequence with $f(n)= \gamma n+o(n)$ if and only if $\gamma\geq\frac{1}{\ln 2}\approx 1.44$.
\end{restatable}

The proof uses the following more general idea of \emph{strongly} piercing sequences.

A sequence $X=(x_1,x_2,\ldots,x_{f(N)})$ is called \emph{strongly $f$-piercing} of \emph{order} $N$ if $\{x_1,x_2,\ldots,x_{f(n)}\}\cap \left[y,y+\frac{1}{n}\right) \neq\emptyset$ holds for every $n=1,2,\ldots, N$ and every real $0\leqslant y \leqslant 1-\frac{1}{n}$, \ie{} $\{x_1,x_2,\ldots,x_{f(n)}\}$ intersects every interval of length $\frac{1}{n}$ that is contained in $[0,1)$.
Clearly, any strongly $f$-piercing sequence is $f$-piercing in the former sense.
Also, we may define analogously infinite sequences with this stronger piercing property.

We will prove in Section 2 that the assertion of \autoref{thm:good} holds for strongly piercing sequences.
Actually, the same result was obtained by de Bruijn and Erd\H{o}s~\cite{BruijnE49} already in 1948.
Next, in Section 3 we explore connections between piercing and strongly piercing sequences more closely, using as a tool the well known \emph{Farey fractions}.
This will allow for deriving the \emph{only if} part of \autoref{thm:good} and the proof of \autoref{thm:s} in Section 4.
In the last section we pose some open problems.

The original problem of Steinhaus is connected to some other, quite famous topics. Indeed, it is visibly related to the well known ``three gaps theorem'' conjectured by him in the mid-1950's, and proved shortly thereafter by many authors, including S\'{o}s~\cite{Sos} and  \'{S}wierczkowski~\cite{Swierczkowski}. The theorem says that for every real number $\alpha>0$ and every natural $N$, the set of points $\{\alpha n~(\modulo 1):n=1,2,\dots,N\}$ splits the unit interval into segments of at most three different lengths. In a less formal way, when you cut a circular cake by turning the knife around by the constant angle, you will always have pieces of at most three different sizes. This illustration resembles the celebrated problem of fair division---a topic of fundamental importance in game theory and economics, studied and popularized by Steinhaus~\cite{SteinhausFD} a decade earlier. Other possible inspirations can be perceived in his work~\cite{SteinhausZM} devoted to quite practical applications of mathematics in some problems of ``measurement''.

\section{Strongly piercing sequences via stick-breaking games}
\label{sec:great}

Consider the following \emph{stick-breaking game}.
At the beginning we have a segment of unit length.
In each subsequent round we choose one of the existing segments and break it into two subsegments.
Before the $k$-th round we have exactly $k$ segments with the total length always equal to one.
Let $M_k$ denote the maximum length of a segment before the $k$-th round.
Our goal in the game is to minimize the value of the limit $\gamma=\limsup (kM_k)$ as $k\rightarrow \infty$.
Notice that $kM_k$ is equal to the ratio $\frac{M_k}{M}$, where $M=\frac{1}{k}$ is the average length of a segment before the $k$-th round.

The idea of studying uniform distribution of a sequence of points via stick-breaking games was explored earlier by Ramshaw~\cite{Ramshaw78}.
Clearly, any infinite sequence $X=(x_1,x_2,\ldots)$ of points in the unit segment determines uniquely a strategy in the stick-breaking game, and vice versa, any stick-breaking strategy gives rise to a unique sequence $X$.
Moreover, it is not hard to check that the sequence $X$ is strongly $(\gamma n+o(n))$-piercing if and only if $\gamma\geq\limsup_{k\rightarrow \infty}(kM_k)$.

Consider now the following stick-breaking strategy.
In each round we choose one of the longest segments and break it into two in the ratio $(1-r):r$, where $0<r\leq \frac{1}{2}$ is a fixed number.
We call this strategy \emph{nonchalant} with \emph{parameter} $r$.

\begin{example}\label{one half}
Let $r=\frac{1}{2}$. In this case all segment lengths in every round are integer powers of $2$,
in particular before $k$-th round, the longest one has length
$M_k=2^{-\lfloor \log_2 k\rfloor}$. For every natural $n$ the number $kM_k$ grows from $1$
for $k=2^n$ to $2-2^{-n}$ for $k=2^{n+1}-1$, hence
$\gamma=\limsup_{k\rightarrow \infty}(kM_k)=2$ and $\delta=\liminf_{k\rightarrow \infty}(kM_k)=1$. Note that the rounds
when there is exactly one longest segment remaining
(e.g. one of length $\frac{1}{8}$ and $14$ of length $\frac{1}{16}$ before round $15$)

are the ones that are important for computing the upper limit, while the immediately following rounds determine the lower limit (e.g. after round $15$ there are $16$ segments of length $\frac{1}{16}$).
\end{example}

In order to obtain lower values of $\gamma$ we will also consider strategies with
a changing parameter. Starting with the segment of length 1 we will
apply the nonchalant strategy with parameter $r_1$, but after some number
of rounds (to be determined more precisely below, but rather large)
we will change the parameter and begin to break the currently longest segment nonchalantly
in the ratio $(1-r_2):r_2$ for some time,  etc. That will be called
the \emph{lazy} strategy corresponding to the sequence $(r_n)$.

\begin{theorem}\label{Theorem Nonchalant}
For every $0<r\leq\frac{1}{2}$ there exists an increasing sequence
$r_n \rightarrow r$ and a corresponding lazy strategy, such that $$\limsup_{k\rightarrow \infty}(kM_k)= -\frac{1}{r \ln r + (1-r) \ln (1-r)}\textrm{.}$$
\end{theorem}
\begin{proof}
First we show the upper and lower limits of the ratio $\frac{M_k}{M}$ as $k\rightarrow\infty$ for a particular class of nonchalant strategies with $r<\frac{1}{2}$.
For that purpose we are going to define a corresponding continuous process of cell splitting in time. In the original setting, every cell (segment) is characterized by its length $d$, and, as the number of round $k$ increases, the segments with the highest $d$ split. Instead of that, we will consider the number $-\ln d$ as the moment in time when the cell of length $d$ splits. Equivalently, considering the state of the game before round $k$, i.e. at the moment when the longest segment has length $M_k$ and is just about to split, we see that the quantity $\tau=-\ln \left(\frac{d}{M_k}\right)$ is the remaining lifetime of any cell of length $d$. Note that all cells with the same highest length $d=M_k$ have, consistently with the above, the remaining lifetimes $\tau=0$, and we treat them as splitting one by one in consecutive rounds, although those rounds are played at the same moment of time.

The process thus begins with one cell that divides at time $t=0$ and each cell, corresponding to a segment of length $d$, splits after the prescribed time into two parts, one of which can be called long-lived, as it corresponds to the subsegment of length $dr$, and will live for the time equal to $t_1=-\ln r$, and the other one short-lived, analogously destined to survive the time $t_2=-\ln(1-r)$, where $t_2<t_1$, as $r<\frac{1}{2}$.

The average length of the interval is currently the arithmetic mean of the expressions $e^{-\tau}$ over all existing cells under the condition that the oldest ones (meaning the longest intervals in the original process) have $\tau=0$ and are in the process of splitting. Denoting the mean by $E$ we have thus $E(e^{-\tau})=E\left(\frac{d}{M_k}\right)=\frac{M}{M_k}$.

Let us now assume that $\frac{t_2}{t_1}$ is a rational number, equal to $\frac{q}{p}$ with $\gcd(p,q)=1$, which lets us define $t=\frac{t_1}{p}=\frac{t_2}{q}$ as the unit of time. In other words the lifetimes and
the splitting moments of all cells are integer multiples of $t$.

We will focus on the number $d_n$ of cells in the $n^{th}$ \emph{generation}, by which we mean the cells that split at the moment $nt$. By the construction, these are exactly the short-lived cells born $q$ units earlier and the long-lived ones born $p$ units earlier, i.e.  $d_n=d_{n-q}+d_{n-p}$ for $n \geq 1$, which together with the initial conditions $d_0=1$ and
$d_{-i}=0$ for $1 \leq i \leq p-1$ gives us the generating function
$$
D(x)=\Sigma d_n x^n = \frac{1}{1-x^p-x^q}\textrm{.}
$$
We claim that $d_n = (C+\oh{1}) \beta^{-n}$ where $C \not= 0$
and $\beta=e^{-t}$ (implying $\beta^p=r$, $\beta^q=1-r$,
and $\beta^p+\beta^q=1$). To that end let us state several
elementary properties of the complex polynomial $W(z)=z^p+z^q-1$.
\begin{itemize}
\item $W$ has no rational roots. Indeed, by the rational root theorem the only candidates are $\pm 1$, but they are not roots of $W$.
\item $W$ increases on $\PosReal$ and therefore has exactly one positive real root, which is equal to $\beta$.
\item $W$ has no multiple roots, either real or not. In fact, the equality $W(z)=W'(z)=0$ implies that
$z^p$ and $z^q$ are rational, which by $\gcd(p,q)=1$ can only be true if $z$ is rational, and that is impossible.\\
\item No complex number $z$ with $|z| \leq \beta$ except $z=\beta$ itself is a root of $W$. This follows from the fact that the inequalities in $1 = |z^p+z^q| \leq |z^p|+|z^q| \leq \beta^p+\beta^q = 1$ become equalities only if $|z| = \beta$ and the ratio $\frac{z^p}{z^q}$ is a positive real, which means that $z$ is a positive real.
\end{itemize}

The above properties imply that the partial fraction decomposition of $D(x)$ has the form:
$$
D(x) = \frac{A}{x-\beta} + \frac{A_2}{x-z_2} + \ldots
+ \frac{A_p}{x-z_p},
$$
where the $z_i$'s are the remaining roots of $W$, and that $d_n = C \beta^{-n} + \oh{\beta^{-n}}$. The last part of the claim, stating that $C \not= 0$ can be proved by contradiction: if $C=0$ then $\lim_{n \rightarrow \infty} d_n \beta^n =0$.
In other words, the total length of the segments of generation $n$ would tend to $0$ as $n \rightarrow \infty$. That is, however, not possible because at most $p+1$ generations of cells/segments are alive at any particular moment, while the total length of all segments is constantly equal to the positive length of the original segment, hence it does not tend to $0$.

It should be noted that we actually proved a little more, namely that the asymptotics $d_n \sim \beta^{-n}$ of the solution of $d_n=d_{n-q}+d_{n-p}$ holds for all non-negative initial conditions, provided that $d_n$ is not the constant $0$ sequence.

In order to compute the average value of $e^{-\tau}$ at the moment when the last cell belonging to the generation
splitting at the time $nt$ is about to disappear (cf. Example~\ref{one half}), it suffices to count how many cells belong to each generation. Here is the detailed bookkeeping: beside the last cell mentioned above, whose remaining lifetime is 0, we have $d_n-1$ newborn long-lived cells with remaining lifetime $pt$, $d_{n-1}$ long lived-cells born one unit earlier with remaining lifetime $(p-1)t$ etc. up to $d_{n-p+1}$ long-lived cells that are going to
split next time, their remaining lifetime being $t$, and $q$ generations of short-lived cells with analogous cardinalities and lifetimes. The result is:
$$
E_n(e^{-\tau}) = \frac{1+ \sum_{i=0}^{p-1}\beta^{p-i}d_{n-i} - \beta^p + \sum_{i=0}^{q-1}\beta^{q-i}d_{n-i} - \beta^q}
{1+ \sum_{i=0}^{p-1}d_{n-i} - 1 + \sum_{i=0}^{q-1}d_{n-i} - 1}.
$$

Knowing that $\frac{d_{n-i}}{d_n}\rightarrow\beta^i$ as $n \rightarrow \infty$ and recalling that $\beta^p+\beta^q=1$, we can write

$$
\lim_{n \rightarrow \infty}E_n(e^{-\tau}) = \frac{\sum_{i=0}^{p-1}\beta^{p-i}\beta^i + \sum_{i=0}^{q-1}\beta^{q-i}\beta^i}
{\sum_{i=0}^{p-1}\beta^i + \sum_{i=0}^{q-1}\beta^i} =
\frac{(1-\beta)(p\beta^p+q\beta^q)}{1-\beta^p+1-\beta^q}=
(1-\beta)(p\beta^p+q\beta^q)=
$$
$$
=-\frac{1-e^{-t}}{t} (r \ln r + (1-r) \ln (1-r)).
$$
Since $E(e^{-\tau})=\frac{M}{M_k}=kM_k$, we have
$$\gamma(r)=\limsup_{k\rightarrow\infty} (kM_k)=-\frac{1}{r \ln r + (1-r) \ln (1-r)}\frac{t}{1-e^{-t}}.$$
In an analogous fashion we obtain the formula
$$\delta(r)=\liminf_{k\rightarrow\infty} (kM_k)=-\frac{1}{r \ln r + (1-r) \ln (1-r)}\frac{t}{e^t-1}.$$

This ends our analysis of nonchalant strategies with parameter $r$ satisfying
$\frac{\ln (1-r)}{\ln r} \in \Rational$. To finish the proof note that for any positive
value of $r \leq \frac{1}{2}$, no matter whether $\frac{\ln (1-r)}{\ln r}$ is rational or not,
we can construct such a sequence $r_n \rightarrow r$ of numbers less than $\frac{1}{2}$, that the sequence $\frac{q_n}{p_n} = \frac{\ln (1-r_n) }{ \ln r_n}$ is an increasing
sequence of rational numbers approximating $\frac{\ln (1-r) }{ \ln r}$ with $p_n\rightarrow\infty$. To be more specific we can use the sequence of binary approximations, meaning that the numbers $p_n$ will be powers of $2$.

If we now apply a lazy strategy in the following way: first use $r=r_1$ for such a long time that
$kM_k$ will be as close to its lower limit
$$-\frac{1}{r_1 \ln r_1 + (1-r_1) \ln (1-r_1)}\frac{t_1}{e^{t_1}-1}$$
as we wish (e.g. closer than $\frac{1}{10}$), then switch to $r=r_2$ and continue till $kM_k$ becomes
less than
$$-\frac{1}{r_2 \ln r_2 + (1-r_2) \ln (1-r_2)}\frac{t_2}{e^{t_2}-1}+\frac{1}{100},$$
etc.
It has to be stressed that the nonchalant strategy with parameter $r_n$ uses as its starting point the set of cells (segments) generated by the strategy with parameter $r_{n-1}$. It is true that the units of time $t_n$ and $t_{n-1}$ are different, and that the expected times of splits may be non-integer numbers when recalculated in the new units, however that does not pose any special problems -- we can say that cells with different fractional parts of their splitting times form separate nonchalant subprocesses with parameter $r_n$ that are activated in cyclic order.
The average length of a segment is then a weighted mean of averages in those processes, and accordingly both $\limsup$ and $\liminf$ of $kM_k$ belong to the interval $[ \delta(r_n), \gamma(r_n) ]$.

Since the sequence $p_n$ tends to infinity, the sequence of time units $t_n=-\frac{\ln r_n}{p_n}$ tends to $0$, which implies $\frac{t_n}{(1-e^{-t_n})} \rightarrow 1$ and $\frac{t_n}{(e^{t_n}-1)} \rightarrow 1$, and thus in the limit we obtain
$$
\limsup_{k\rightarrow\infty} (kM_k) = \liminf_{k\rightarrow\infty} (kM_k) = -\frac{1}{r \ln r + (1-r) \ln (1-r)},
$$
as asserted.
\end{proof}

The following result due to de Bruijn and Erd\H{o}s~\cite{BruijnE49} can now be seen as the special case
of Theorem~\ref{Theorem Nonchalant} for $r=\frac{1}{2}$.

\begin{corollary}[de Bruijn and Erd\H{o}s~\cite{BruijnE49}]\label{Corollary dBE}
There exists a strongly $f$-piercing sequence with $f(n)=\frac{1}{\ln 2}n+o(n)$.
\end{corollary}

The original proof of this statement in~\cite{BruijnE49} provides an elegant explicit example of the desired sequence.
Namely, it is proved there that the sequence of fractional parts of the numbers $x_i=\log_2(2i-1)$, $i=1,2,\ldots$, has the claimed strong piercing property.
It is also proved in~\cite{BruijnE49} that the constant $\frac{1}{\ln 2}$ in this result is optimal.
For the sake of completeness, we present the proof of de Bruijn and Erd\H{o}s of this fact.
The key Lemma~\ref{lem:gamma2n} implying it will also play a role in our proof of Theorem~\ref{thm:s}.

For any positive integer $n$, let $H_n$ denote the $n$-th harmonic number, \ie{} $H_n=\sum_{i=1}^n\frac{1}{i}$.
\begin{lemma}[de Bruijn and Erd\H{o}s~\cite{BruijnE49}]\label{lem:gamma2n}
For every integer $N\geq 2$, define $\gamma_N = \frac{1}{H_{2N}-H_{N-1}}$.
For every sequence $X=\seq{x_1,x_2,\ldots,x_{2N}}$ of reals in $[0,1)$,
for every $N \leq n \leq 2N$, let $Y^n$ be the sequence of numbers $0,1,x_1,x_2,\ldots,x_n$ arranged in increasing order, \ie{} $Y^n=\seq{y^n_1,y^n_2,\ldots,y^n_{n+2}} = \sort\seq{0,1,x_1,x_2,\ldots,x_{n}}$, and $b_n$ be the maximum difference between any two consecutive elements of $Y^n$, \ie{} $b_n=\max_{i=1,2,\ldots,n+1} \brac{y^n_{i+1} - y^n_{i}}$.
Then, for at least one $N \leq n \leq 2N$, we have $b_n \geq \gamma_N\cdot\frac{1}{n}$.
\end{lemma}
\begin{proof}
Assume to the contrary that $b_n < \gamma_N\cdot\frac{1}{n}$ holds for every $N \leq n \leq 2N$.
Let $d_1 \leq d_2 \leq \ldots \leq d_{N+1}$ be the sequence of differences between consecutive elements of $Y^N$ arranged in increasing order, \ie{} $\seq{d_1,d_2,\ldots,d_{N+1}} = \sort\seq{y^N_2-y^N_1,y^N_3-y^N_2,\ldots,y^N_{N+2}-y^N_{N+1}}$.
First, observe that $d_1+d_2+\ldots+d_{N+1} = 1$.
Second, as adding one point to the sequence splits only one difference between two consecutive elements, we have $b_N = d_{N+1}$, $b_{N+1} \geq d_{N}$, $b_{N+2} \geq d_{N-1}$, $\ldots$, $b_{2N} \geq d_{1}$.
Thus, we get
$$1 = \sum_{i=1}^{N+1} d_i \leq \sum_{n=N}^{2N} b_n < \gamma_N\sum_{n=N}^{2N}\frac{1}{n} = \gamma_N \brac{H_{2N}-H_{N-1}}\textrm{,}$$
which contradicts the choice of $\gamma_N$.
\end{proof}

\begin{theorem}[de Bruijn and Erd\H{o}s~\cite{BruijnE49}]\label{Theorem dBE}
There does not exist a strongly $f$-piercing sequence with $f(n)=\gamma n+o(n)$ for any real number $\gamma<\frac{1}{\ln 2}$.
\end{theorem}

\begin{proof}

Let $X$ be an infinite strongly $f$-piercing sequence with $f(n')=\gamma{}n'+\oh{n'}$ and $\gamma < \frac{1}{\ln 2}$.
As $H_{2N}-H_{N-1} \to \ln 2$ when $N \to \infty$, we can choose $N_0$ large enough so that
$$f(n') \leq \frac{n'}{H_{2f(N_0)}-H_{f(N_0)-1}+\frac{1}{N_0}}$$
for every $n' \geq N_0$.
For every $f(N_0) \leq n \leq 2f(N_0)$ we have that $Z_n=\seq{X_1,X_2,\dots,X_n}$ contains as prefix a strongly $f$-piercing sequence of order
$$\floor{n\left(H_{2f(N_0)}-H_{f(N_0)-1}+\frac{1}{N_0}\right)}>n(H_{2f(N_0)}-H_{f(N_0)-1}).$$
Lemma~\ref{lem:gamma2n} states that for some $f(N_0) \leq n \leq 2f(N_0)$ we have
that the maximum distance between consecutive elements of $\sort\seq{0,1,X_1,X_2,\ldots,X_n}$ equals at least
$$\frac{1}{(H_{2f(N_0)}-H_{f(N_0)-1})n},$$
which gives a contradiction.

\end{proof}

\section{Piercing sequences and Farey fractions}
\label{sec:good}

Our aim in this section is to prove that, given an $f$-piercing sequence with $f(n)=\gamma n+o(n)$, one may construct a strongly $g$-piercing sequence with $g(n)=\gamma'n+o(n)$, where the constant $\gamma'$ is arbitrarily close to $\gamma$.
By Theorem~\ref{Theorem dBE}, this proves \autoref{thm:good}.

Our main tool are the well known sequences of Farey fractions.
Recall that the sequence of \emph{Farey fractions} of \emph{order} $p$ consists of all irreducible fractions in the unit interval $[0,1]$ whose denominators do not exceed $p$ (see~\cite{NivenZ72}).
We use the term \emph{Farey points} for the points on the real line corresponding to Farey fractions.

We need the following notation and terminology.
For any integers $0<n\leq m$, let $$\FP_n^m = \set{\frac{a}{p} : n \leq p \leq m, 0 \leq a \leq p}$$ denote the set of Farey points that can be expressed by a fraction with denominator between $n$ and $m$.
We say that two points $\frac{a}{p}, \frac{b}{q} \in \FP_n^m$ are \emph{consecutive} in $\FP_n^m$,
$\frac{a}{p}$ is \emph{previous} for $\frac{b}{q}$ in $\FP_n^m$, and $\frac{b}{q}$ is \emph{next} for $\frac{a}{p}$ in $\FP_n^m$,
when $\frac{a}{p} < \frac{b}{q}$ and there is no other $\frac{c}{r} \in \FP_n^m$ with $\frac{a}{p} < \frac{c}{r} < \frac{b}{q}$.
Similarly, let $$\FI_n^m = \set{\lbrac{\frac{a}{p},\frac{a+1}{p}} : n \leq p \leq m, 0 \leq a \leq p-1}$$ denote the set of Farey intervals defined by Farey points in $\FP_n^m$.
Lastly, let $$\AI^n = \set{\lbrac{y,y+\frac{1}{n}} : 0 \leq y \leq \frac{n-1}{n}}$$ denote the set of all intervals of length $\frac{1}{n}$ that are contained in $\lbrac{0,1}$.

We begin our investigation with a lemma, that captures the following intuition. Given any two reals $\alpha>\beta>1$ and $n$ big enough,
a large majority of intervals of length $\frac{1}{n}$ in $\lbrac{0,1}$
contain one of the Farey intervals with denominator between $\beta{}n$ and $\alpha{}n$.
There are only some exceptional regions in $[0,1)$, but not too many.

\begin{lemma}\label{lem:intervals}
Let $W \geqslant 2$ be a fixed integer, $\alpha = \frac{W+1}{W-1}$, and $\beta = \frac{W}{W-1}$.
For every $N > N_0 = N_0(W) = 2W^3$, and every interval $\lbrac{y,y+\frac{1}{N}} \in \AI^N$, we have that either
\begin{itemize}
\item $\lbrac{y,y+\frac{1}{N}}$ contains one of Farey intervals $\lbrac{\frac{c}{r},\frac{c+1}{r}} \in \FI_{\ceil{\beta{}N}}^{\floor{\alpha{}N}}$,
or
\item $\norm{y-\frac{b}{q}} < \frac{W+1}{N}$ for some Farey point $\frac{b}{q} \in \FP_1^{W-1}$.
\end{itemize}
\end{lemma}
\begin{proof}
For any $y \in \lbrac{0,\frac{N-1}{N}}$, we say that a Farey point $\frac{c}{r}$ in $\FP_{\ceil{\beta{}N}}^{\floor{\alpha{}N}}$ is a \emph{valid cover} of $y$ when $y < \frac{c}{r} \leq y + \frac{1}{WN}$.
Observe that for any valid cover $\frac{c}{r}$ of $y$, we get
$$y < \frac{c}{r} < \frac{c+1}{r} \leq y + \frac{1}{WN} + \frac{1}{r} \leq y + \frac{1}{WN} + \frac{1}{\beta{}N} = y + \frac{1+(W-1)}{WN} = y + \frac{1}{N}\textrm{,}$$
and the interval $\lbrac{y,y+\frac{1}{N}}$ contains Farey interval $\lbrac{\frac{c}{r},\frac{c+1}{r}}$.
Further, if $\frac{c}{r}$ is a valid cover of $\frac{a}{p}$, then $\frac{c}{r}$ is a valid cover of all $y$ in $\lbrac{\frac{a}{p},\frac{c}{r}}$.
Thus, we are interested in finding all pairs of Farey points consecutive in $\FP_{\ceil{\beta{}N}}^{\floor{\alpha{}N}}$ that are at distance larger than $\frac{1}{WN}$.
We show that for every $\frac{a}{p} \in \FP_{\ceil{\beta{}N}}^{\floor{\alpha{}N}}$ either
\begin{itemize}
\item there is a valid cover $\frac{c}{r}$ of $\frac{a}{p}$, or
\item there is a Farey point $\frac{b}{q} \in \FP_1^{W-1}$ with $\norm{\frac{a}{p}-\frac{b}{q}} < \frac{W}{N}$.
\end{itemize}
Now, fix any $\frac{b}{q} \in \FP_1^{W-1}$ and observe that for an $\frac{a}{p} \in \FP_{\ceil{\beta{}N}}^{\floor{\alpha{}N}}$:
\begin{itemize}
\item If $\frac{a}{p} < \frac{b}{q}$ and $a \geq \frac{b(W-1)p-p-q}{q(W-1)}$, then for $\frac{c}{r}=\frac{a+b}{p+q}$ we have:
\begin{align}
b(W-1)p - p - q & \leq a(W-1)q
& \textrm{/$+a(W-1)p$}  \nonumber\\
(a+b)(W-1)p & \leq a(W-1)(p+q) + (p+q)
& \textrm{/$\div (W-1)p(p+q)$}  \nonumber\\
\frac{a+b}{p+q} & \leq \frac{a}{p} + \frac{1}{(W-1)p},~~~~~~~~~~~~~~~~~{\rm and~thus:} \nonumber\\
\frac{a}{p} < \frac{c}{r} = \frac{a+b}{p+q} & \leq \frac{a}{p} + \frac{1}{(W-1)p} \leq \frac{a}{p} + \frac{1}{(W-1)\beta{}N} = \frac{a}{p} + \frac{1}{WN}\textrm{.} \label{ValidCoverJust}
\end{align}
Hence, $\frac{c}{r}$ is a valid cover of $\frac{a}{p}$, if $p+q = r \leq \floor{\alpha{}N}$.
\item Similarly, if $\frac{a}{p} > \frac{b}{q}$ and $a \leq \frac{b(W-1)p+p-q}{q(W-1)}$, then for $\frac{c}{r}=\frac{a-b}{p-q}$ we have:
\begin{align*}
\frac{a}{p} < \frac{c}{r} = \frac{a-b}{p-q} & \leq \frac{a}{p} + \frac{1}{(W-1)p}\leq
\frac{a}{p} + \frac{1}{WN}\textrm{.}
\end{align*}
Thus, $\frac{c}{r}$ is a valid cover of $\frac{a}{p}$, if $p-q = r \geq \ceil{\beta{}N}$.
\end{itemize}
Inspired by these observations, we say that a Farey point $\frac{b}{q} \in \FP_1^{W-1}$ \emph{provides cover} of a Farey point $\frac{a}{p} \in \FP_{\ceil{\beta{}N}}^{\floor{\alpha{}N}}$, if either
\begin{itemize}
\item $\frac{a}{p} < \frac{b}{q}$, $p+q\leq\floor{\alpha{}N}$, and $a \geq \frac{b(W-1)p-p-q}{q(W-1)}$, or
\item $\frac{a}{p} > \frac{b}{q}$, $p-q\geq\ceil{\beta{}N}$, and $a \leq \frac{b(W-1)p+p-q}{q(W-1)}$.
\end{itemize}
\begin{figure}
\centering
\begin{tikzpicture}[scale=0.12]
\clip (-9.450000000000001,56.55) rectangle (113.4, 101.4);
\draw[<->,ultra thick] (94.5,60) node[below]{$a$} --(0,60)--(0,94.5) node[left]{$p$};
\draw[-] (0,90) node[left]{$\alpha{}N$} --(90,90);
\draw[-] (0,75) node[left]{$\beta{}N$}--(90,75);
\draw[->,thick,red] (0,0)--(0.0,94.5) node[right]{$\frac{0}{1}$};
\filldraw[fill opacity=0.5,red,fill=red] (0,75)--(0,89)--(0.0,89)--(0.0,90)--(22.2,90)--(18.746666666666666,76)--(0.0,76)--(0.0,75)--cycle;
\draw[->,thick,black] (10.236666666666666,83.5)--(10.236666666666666,82.5);
\draw[->,thick,green] (0,0)--(23.625,94.5) node[right]{$\frac{1}{4}$};
\filldraw[fill opacity=0.5,green,fill=green] (13.8125,75)--(15.838333333333335,86)--(21.5,86)--(22.5,90)--(27.825,90)--(24.424166666666665,79)--(19.75,79)--(18.75,75)--cycle;
\draw[->,thick,black] (16.975208333333335,78.5)--(17.975208333333335,82.5);
\draw[->,thick,black] (24.124791666666663,86.5)--(23.124791666666663,82.5);
\draw[->,thick,blue] (0,0)--(31.5,94.5) node[right]{$\frac{1}{3}$};
\filldraw[fill opacity=0.5,blue,fill=blue] (18.5,75)--(21.46,87)--(29.0,87)--(30.0,90)--(37.2,90)--(32.24,78)--(26.0,78)--(25.0,75)--cycle;
\draw[->,thick,black] (22.99,79.5)--(23.99,82.5);
\draw[->,thick,black] (31.859999999999996,85.5)--(30.859999999999996,82.5);
\draw[->,thick,yellow] (0,0)--(47.25,94.5) node[right]{$\frac{1}{2}$};
\filldraw[fill opacity=0.5,yellow,fill=yellow] (27.875,75)--(32.70666666666666,88)--(44.0,88)--(45.0,90)--(55.95,90)--(47.86833333333334,77)--(38.5,77)--(37.5,75)--cycle;
\draw[->,thick,black] (35.02041666666666,80.5)--(36.02041666666666,82.5);
\draw[->,thick,black] (47.32958333333333,84.5)--(46.32958333333333,82.5);
\draw[->,thick,magenta] (0,0)--(63.0,94.5) node[right]{$\frac{2}{3}$};
\filldraw[fill opacity=0.5,magenta,fill=magenta] (43.5,75)--(50.459999999999994,87)--(58.0,87)--(60.0,90)--(67.2,90)--(58.24,78)--(52.0,78)--(50.0,75)--cycle;
\draw[->,thick,black] (49.489999999999995,79.5)--(51.489999999999995,82.5);
\draw[->,thick,black] (60.36,85.5)--(58.36,82.5);
\draw[->,thick,cyan] (0,0)--(70.875,94.5) node[right]{$\frac{3}{4}$};
\filldraw[fill opacity=0.5,cyan,fill=cyan] (51.3125,75)--(58.83833333333334,86)--(64.5,86)--(67.5,90)--(72.825,90)--(63.92416666666667,79)--(59.25,79)--(56.25,75)--cycle;
\draw[->,thick,black] (56.22520833333333,78.5)--(59.22520833333333,82.5);
\draw[->,thick,black] (67.37479166666667,86.5)--(64.37479166666667,82.5);
\draw[->,thick,red] (0,0)--(94.5,94.5) node[right]{$\frac{1}{1}$};
\filldraw[fill opacity=0.5,red,fill=red] (56.0,75)--(66.45333333333333,89)--(89.0,89)--(90.0,90)--(90,90)--(76,76)--(76.0,76)--(75.0,75)--cycle;
\draw[->,thick,black] (71.11333333333333,81.5)--(72.11333333333333,82.5);
\end{tikzpicture}
\caption{$W=5$, $N=60$, $\alpha{}N=90$, $\beta{}N=75$. Farey points in $\FP_1^4 = \set{\frac{0}{1},\frac{1}{4},\frac{1}{3},\frac{1}{2},\frac{2}{3},\frac{3}{4},\frac{1}{1}}$ provide cover of regions of $\FP_{75}^{90}$. Arrows represent the direction of $\pm \frac{b}{q}$. For every point $\frac{a}{p}$ in the blue region to the right of $\frac{1}{3}$ arrow, the point $\frac{a-1}{p-3}$ is a valid cover of $\frac{a}{p}$. }
\label{fig:moves}
\end{figure}
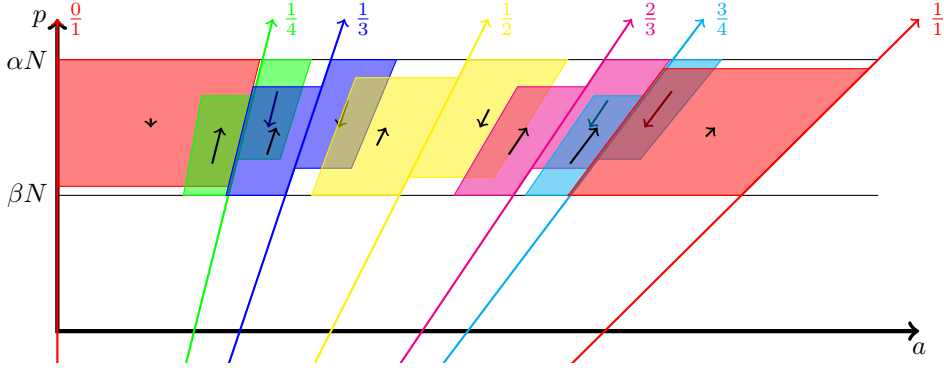

See \autoref{fig:moves} for an example, where for $W=5$, $N=60$ (which does not satisfy the condition $N > 2W^3$, but allows for clearer drawing), and each $\frac{b}{q} \in \FP_1^{W-1}$, there are depicted regions of those $\frac{a}{p}$ for which $\frac{b}{q}$ provides cover.
It is not a coincidence that in \autoref{fig:moves}, regions covered by any two consecutive Farey points overlap.

Consider any two consecutive Farey points $\frac{b_1}{q_1} < \frac{b_2}{q_2}$ in $\FP_1^{W-1}$.
As it is well known that $q_1+q_2 \geq W$ and $b_2q_1 - b_1q_2 = 1$ (see~\cite{NivenZ72}), for any integer $p \geq \beta{}N > 2W^3 > 2(W-1)q_1q_2$, we have:
\begin{align}
q_2 + q_1 & \geq (b_2q_1-b_1q_2)(W-1)+1
& \textrm{/$\cdot p$}  \nonumber\\
b_1(W-1)pq_2 + pq_2 & \geq b_2(W-1)pq_1 - pq_1 + p
& \textrm{/$-q_1q_2$}  \nonumber\\
b_1(W-1)pq_2 + pq_2 -q_1q_2 & > b_2(W-1)pq_1 - pq_1 -q_1q_2 + 2(W-1)q_1q_2
& \textrm{/$\div (W-1)q_1q_2$}  \nonumber\\
\frac{b_1(W-1)p + p - q_1}{(W-1)q_1} & > \frac{b_2(W-1)p - p - q_2}{(W-1)q_2} + 2\textrm{.} \label{GlueInequality}
\end{align}

Thus there is at least one integer $a$ such that $\frac{b_2(W-1)p-p-q_2}{q_2(W-1)} \leq a \leq \frac{b_1(W-1)p+p-q_1}{q_1(W-1)}$.
We say then that $\frac{a}{p}$ \emph{glues} $\frac{b_1}{q_1}$ with $\frac{b_2}{q_2}$.
Note its existence in particular implies that  for every $\frac{a}{p}\in \FP_{\ceil{\beta{}N}}^{\floor{\alpha{}N}} \cap (\frac{b_1}{q_1},\frac{b_2}{q_2})$ at least one of the inequalities: $a \geq \frac{b_2(W-1)p-p-q_2}{q_2(W-1)}$ or
$a \leq \frac{b_1(W-1)p+p-q_1}{q_1(W-1)}$ holds.

Thus if only $p \in \sbrac{\ceil{\beta{}N}+q_1,\floor{\alpha{}N}-q_2} \subseteq \sbrac{\ceil{\beta{}N}+W-1,\floor{\alpha{}N}-W+1}$, we are certain that

there is at least one $\frac{b}{q} \in \FP_1^{W-1}$ that provides cover of $\frac{a}{p}$, which implies that
$\frac{a}{p}$ has a valid cover.

The last remaining problem is to find valid covers of Farey points
$\frac{a}{p}\in\FP_{\ceil{\beta{}N}}^{\floor{\alpha{}N}}$
whose denominators are near $\beta{}N$ or $\alpha{}N$
and which are not too close to $\FP_1^{W-1}$ themselves, i.e.
$\frac{a}{p}\in [\frac{b'}{q'}+\frac{W}{N},\frac{b}{q}-\frac{W}{N}]$
for some consecutive $\frac{b'}{q'},\frac{b}{q}$ in $\FP_1^{W-1}$.

Suppose first that $p \geq \floor{\alpha{}N}-q+1$. Let $\frac{c}{r}\in\FP_{\ceil{\beta{}N}}^{\floor{\alpha{}N}}$ be such that $\ceil{\beta{}N} \leq r < \ceil{\beta{}N}+q$, $r \equiv p~(\modulo q)$, and $c = \floor{\frac{ar}{p}}$.
If $\frac{b'}{q'}$ provides cover of $\frac{a}{p}$, we are done. We may thus assume this is not the case, in particular $\frac{a}{p}$ does not glue $\frac{b'}{q'}$ with $\frac{b}{q}$, and hence by~(\ref{GlueInequality}),
\begin{align}
a &> \frac{b(W-1)p-p-q}{q(W-1)}+2& \nonumber\\
&= \left(\frac{b(W-1)p-p-\frac{qp}{r}}{q(W-1)}+\frac{p}{r}\right)
+\frac{\frac{p}{r}(2-W)+2W-3}{W-1}& \nonumber\\
& >    \frac{b(W-1)p-p-\frac{qp}{r}}{q(W-1)}+\frac{p}{r},\label{FirstAInequality}
\end{align}
where the last inequality follows by the fact that $\frac{p}{r}\leq \frac{\alpha{}}{\beta{}}=\frac{W+1}{W}< 2$.
Note that $\frac{c}{r} \leq \frac{a}{p} < \frac{c+1}{r}$. Thus by~(\ref{FirstAInequality}),
\begin{align*}
c & > \frac{ar}{p}-1
> \frac{r}{p}\left(\frac{b(W-1)p-p-\frac{qp}{r}}{q(W-1)}+\frac{p}{r}\right)-1 =\frac{b(W-1)r-r-q}{q(W-1)}, &
\end{align*}
and hence $\frac{b}{q}$ provides cover of $\frac{c}{r}$.

Now, let $t=\frac{p-r}{q}$, and consider the sequence $\frac{c_i}{r_i} = \frac{c+bi}{r+qi}$, for $i = 0,1,\ldots{},\frac{p-r}{q}=t$ and observe that for each $i\geq1$, $\frac{c_i}{r_i}>\frac{c_{i-1}}{r_{i-1}}$ and, since $\frac{b}{q}$ provides cover of $\frac{c}{r}$, then by~(\ref{ValidCoverJust}),
\begin{align*}
\frac{1}{WN} &\geq \frac{c+b}{r+q} - \frac{c}{r} =
\frac{rb-qc}{(r+q)r} \geq
\frac{rb-qc}{(r+qi)(r+q(i-1))} =
\frac{c_i}{r_i} - \frac{c_{i-1}}{r_{i-1}}, &
\end{align*}
and hence $\frac{c_i}{r_i}$ is a valid cover of $\frac{c_{i-1}}{r_{i-1}}$.
See \autoref{fig:border} for an example with $W=5, N=300, \frac{a}{b}=\frac{2}{3}, \frac{a}{p}=\frac{285}{448}$, and $\frac{c_i}{r_i} = \seq{\frac{238}{376},\frac{240}{379},\ldots{},\frac{286}{448}}$.
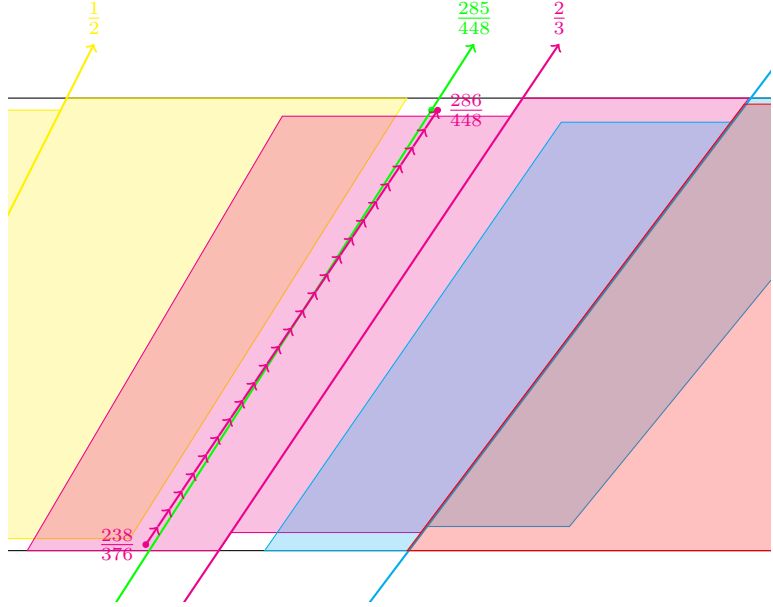
\begin{figure}
\centering
\begin{tikzpicture}[scale=0.08]
\clip (215.35714285714286,366.6) rectangle (340.7142857142857, 475.8);
\draw[<->,ultra thick] (321.42857142857144,300) node[below]{$a$} --(0,300)--(0,459.0) node[left]{$p$};
\draw[-] (0,450) node[left]{$\alpha{}N$} --(450,450);
\draw[-] (0,375) node[left]{$\beta{}N$}--(450,375);
\draw[->,thick,red] (0,0)--(0.0,459.0) node[above]{$\frac{0}{1}$};
\filldraw[fill opacity=0.25,red,fill=red] (0,375)--(0,449)--(0.0,449)--(0.0,450)--(112.19999999999999,450)--(93.74933333333333,376)--(0.0,376)--(0.0,375)--cycle;
\draw[->,thick,green] (0,0)--(114.75,459.0) node[above]{$\frac{1}{4}$};
\filldraw[fill opacity=0.25,green,fill=green] (70.0625,375)--(83.32766666666666,446)--(111.5,446)--(112.5,450)--(140.32500000000002,450)--(118.18483333333334,379)--(94.75,379)--(93.75,375)--cycle;
\draw[->,thick,blue] (0,0)--(153.0,459.0) node[above]{$\frac{1}{3}$};
\filldraw[fill opacity=0.25,blue,fill=blue] (93.5,375)--(111.452,447)--(149.0,447)--(150.0,450)--(187.2,450)--(157.248,378)--(126.0,378)--(125.0,375)--cycle;
\draw[->,thick,yellow] (0,0)--(229.5,459.0) node[above]{$\frac{1}{2}$};
\filldraw[fill opacity=0.25,yellow,fill=yellow] (140.375,375)--(167.70133333333334,448)--(224.0,448)--(225.0,450)--(280.95,450)--(235.37366666666665,377)--(188.5,377)--(187.5,375)--cycle;
\draw[->,thick,magenta] (0,0)--(306.0,459.0) node[above]{$\frac{2}{3}$};
\filldraw[fill opacity=0.25,magenta,fill=magenta] (218.5,375)--(260.452,447)--(298.0,447)--(300.0,450)--(337.2,450)--(283.248,378)--(252.0,378)--(250.0,375)--cycle;
\draw[->,thick,cyan] (0,0)--(344.25,459.0) node[above]{$\frac{3}{4}$};
\filldraw[fill opacity=0.25,cyan,fill=cyan] (257.5625,375)--(306.32766666666663,446)--(334.5,446)--(337.5,450)--(365.325,450)--(307.6848333333333,379)--(284.25,379)--(281.25,375)--cycle;
\draw[->,thick,red] (0,0)--(459.0,459.0) node[above]{$\frac{1}{1}$};
\filldraw[fill opacity=0.25,red,fill=red] (281.0,375)--(336.4506666666667,449)--(449.0,449)--(450.0,450)--(450,450)--(376,376)--(376.0,376)--(375.0,375)--cycle;
\node[draw,fill,circle,green,inner sep=0.75pt,outer sep=0] at (285,448) {};
\draw[->,thick,green] (0,0)--(291.99776785714283,459.0) node[above]{$\frac{285}{448}$};
\node[draw,fill,circle,inner sep=0.75pt,outer sep=0,magenta] at (286,448) {};
\node[right,magenta] at (286,448) {$\frac{286}{448}$};
\draw[<-,thick,magenta] (286,448)--(284,445);
\draw[<-,thick,magenta] (284,445)--(282,442);
\draw[<-,thick,magenta] (282,442)--(280,439);
\draw[<-,thick,magenta] (280,439)--(278,436);
\draw[<-,thick,magenta] (278,436)--(276,433);
\draw[<-,thick,magenta] (276,433)--(274,430);
\draw[<-,thick,magenta] (274,430)--(272,427);
\draw[<-,thick,magenta] (272,427)--(270,424);
\draw[<-,thick,magenta] (270,424)--(268,421);
\draw[<-,thick,magenta] (268,421)--(266,418);
\draw[<-,thick,magenta] (266,418)--(264,415);
\draw[<-,thick,magenta] (264,415)--(262,412);
\draw[<-,thick,magenta] (262,412)--(260,409);
\draw[<-,thick,magenta] (260,409)--(258,406);
\draw[<-,thick,magenta] (258,406)--(256,403);
\draw[<-,thick,magenta] (256,403)--(254,400);
\draw[<-,thick,magenta] (254,400)--(252,397);
\draw[<-,thick,magenta] (252,397)--(250,394);
\draw[<-,thick,magenta] (250,394)--(248,391);
\draw[<-,thick,magenta] (248,391)--(246,388);
\draw[<-,thick,magenta] (246,388)--(244,385);
\draw[<-,thick,magenta] (244,385)--(242,382);
\draw[<-,thick,magenta] (242,382)--(240,379);
\draw[<-,thick,magenta] (240,379)--(238,376);
\node[draw,fill,circle,inner sep=0.75pt,outer sep=0,magenta] at (238,376) {};
\node[left,magenta] at (238,376) {$\frac{238}{376}$};
\end{tikzpicture}
\caption{
$W=5$, $N=300$, $\alpha{}N=450$, $\beta{}N=375$. Farey point $\frac{2}{3}$ does not provide cover of $\frac{285}{448}$ (as $448+3>\alpha{}N$),
but line parallel to $\frac{2}{3}$ through points $\frac{238}{376}$, and $\frac{286}{448}$ crosses line $\frac{285}{448}$ somewhere above $\ceil{\beta{}N}$.
}
\label{fig:border}
\end{figure}
For $i=0$, we have $\frac{c_0}{r_0} = \frac{c}{r} \leq \frac{a}{p}$.
For $i=t=\frac{p-r}{q}$, we have $\frac{c_t}{r_t} = \frac{c+bt}{p}$.
Observe that for $N > 2W^3$, we have

\begin{align}
p-r > (\alpha{}N -q) - (\beta{}N+q) = \frac{1}{W-1}N-2q \geq \frac{1}{W-1}N-2(W-1) > \frac{N}{W}, \label{P-R_Calc}
\end{align}

and therefore:
\begin{align*}
\frac{a}{p}  & \leq \frac{b}{q} - \frac{W}{N} \\
\frac{a}{p}  & < \frac{b}{q} - \frac{1}{p-r} \\
aq(p-r) & < bp(p-r) - pq \\
aqp & < arq - pq + b(p-r)p \\
a & < \frac{ar-p}{p} + b\cdot\frac{p-r}{q} < c + bt = c_t \\
\frac{a}{p} & < \frac{c_t}{p} = \frac{c_t}{r_t}\textrm{.}
\end{align*}

So, $\frac{c_0}{r_0} \leq \frac{a}{p} < \frac{c_t}{r_t}$, and hence there is some $i \in \set{1,2,\ldots,t}$, such that $\frac{c_{i-1}}{r_{i-1}} \leq \frac{a}{p} < \frac{c_{i}}{r_{i}}$.
As $\frac{c_{i}}{r_{i}}$ is a valid cover of $\frac{c_{i-1}}{r_{i-1}}$, it is also a valid cover of $\frac{a}{p}$.

It remains to discuss the situation when $p$ is close to $\beta{}N$.
By~(\ref{GlueInequality}) and the case discussed above we may thus assume that
$p \leq \ceil{\beta{}N}+q'-1$.
This time we choose $\frac{c}{r}\in\FP_{\ceil{\beta{}N}}^{\floor{\alpha{}N}}$ with
$\floor{\alpha{}N}-q' < r \leq \floor{\alpha{}N}$, $r \equiv p~(\modulo q')$, and $c = a+1+b'\cdot\frac{r-p}{q'}$.
Similarly as in~(\ref{P-R_Calc}), $r-p>\frac{N}{W}$.
Since additionally $\frac{a}{p}-\frac{b'}{q'}\geq\frac{W}{N}$, then we have:
\begin{align}
1 & < \left(\frac{a}{p}-\frac{b'}{q'}\right)(r-p) & \nonumber\\
1+\frac{b'}{q'}(r-p) & < \frac{a}{p}(r-p) &\nonumber\\
p+b'p\frac{r-p}{q'} & < a(r-p) &\nonumber\\
pa+p+b'p\frac{r-p}{q'} & < ar & \textrm{/$\div (pr)$}\nonumber\\
\frac{c}{r}=\frac{a+1+b'\frac{r-p}{q'}}{r} & < \frac{a}{p}& \label{CRSmall2}
\end{align}
and moreover, as $\frac{a}{p}>\frac{b'}{q'}=\frac{b'\frac{r-p}{q'}}{q'\frac{r-p}{q'}}$, then
\begin{align*}
\frac{b'}{q'}<&\frac{a+b'\frac{r-p}{q'}}{p+q'\frac{r-p}{q'}} < \frac{a+1+b'\cdot\frac{r-p}{q'}}{r} = \frac{c}{r}.&
\end{align*}
Now, consider the sequence $\frac{c_i}{r_i} = \frac{c-b'i}{r-q'i}$, for $i = 0,1,\ldots{},t=\frac{r-p}{q'}$, and note that
$\frac{c_t}{r_t} = \frac{a+1}{p} > \frac{a}{p}$, while
by~(\ref{CRSmall2}), $\frac{c_0}{r_0} =\frac{c}{r} < \frac{a}{p}$.
Thus
for some $i \in \set{1,2,\ldots,t}$ we have $\frac{c_{i-1}}{r_{i-1}} \leq \frac{a}{p} < \frac{c_i}{r_i}$.
Analogously as in the previous case, by~(\ref{GlueInequality}), we may assume that $a < \frac{b'(W-1)p+p-q'}{q'(W-1)}-2 < \frac{b'(W-1)p+p-q'}{q'(W-1)}-1$, and thus, as $r_{i-1}-q'\geq p \geq \ceil{\beta{}N}$,
\begin{align*}
\frac{c_{i-1}}{r_{i-1}} < & \frac{a+1}{p} < \frac{b'(W-1)+1-\frac{q'}{p}}{q'(W-1)} <
\frac{b'(W-1)+1-\frac{q'}{r_{i-1}}}{q'(W-1)},
\end{align*}
and thus $\frac{b'}{q'}$ provides cover of $\frac{c_{i-1}}{r_{i-1}}$, which means that $\frac{1}{NW}\geq \frac{c_{i-1}-b'}{r_{i-1}-q'}-\frac{c_{i-1}}{r_{i-1}} = \frac{c_{i}}{r_{i}}-\frac{c_{i-1}}{r_{i-1}}$, and thus $\frac{c_{i}}{r_{i}}$ is a valid cover of $\frac{a}{p}$.

We have found a valid cover of every Farey point in $\FP_{\ceil{\beta{}N}}^{\floor{\alpha{}N}}$ except those that are at distance smaller than $\frac{W}{N}$ from some Farey point in $\FP_1^{W-1}$.
Since for every $\frac{b}{q}\in\FP_1^{W-1}$ there must be some Farey point in $\FP_{\ceil{\beta{}N}}^{\floor{\alpha{}N}}$ at distance at most $\frac{1}{\floor{\alpha{}N}}<\frac{1}{N}$ (in any direction) from $(\frac{b}{q}-\frac{W}{N},\frac{b}{q}+\frac{W}{N})$,
this yields a valid cover of every real point $y\in[0,1)$ except possibly those that are at distance smaller than $\frac{W+1}{N}$ from some Farey point in $\FP_1^{W-1}$ and ends the proof.
\end{proof}

Observe that there are only $\Oh{W^2}$ Farey points in $\FP_1^{W-1}$, and the length of the union of all intervals in $\AI^n$ that do not include any interval from $\FI_{\ceil{\beta{}N}}^{\floor{\alpha{}N}}$ is $\frac{\Oh{W^3}}{N}$.
This length tends to zero as $N$ tends to infinity. The next lemma exploits this fact to construct a strongly piercing sequence from a piercing one having similar parameters.

\begin{lemma}\label{lem:goodgreat}
Let $W \geqslant 2$ be a fixed integer, and let $\alpha = \frac{W+1}{W-1}$.
For every function $f(n) = \gamma{}n + \oh{n}$, there
exist $N_0$ and a function $g(n) = \alpha^2\gamma{}n + \oh{n}$ such that for every $N\geq N_0$
if there is an $f$-piercing sequence $X$ of order $\floor{\alpha{}N}$, then there is a strongly $g$-piercing sequence $Z$ of order $N$.
\end{lemma}

\begin{proof}
Let $\beta = \frac{W}{W-1}$, and choose $N_0 > 2W^3$ so that $f(N) < \alpha\gamma N$ for all $N\geq N_0$.
We define $g(n) = \ceil{\alpha^2\gamma{}n + N_0 + 5W^3(\log_2n+ 1)}$.
Assume that $X=(x_1,x_2,\ldots{},x_{f(\floor{\alpha{}N})})$ is an $f$-piercing sequence of order $\floor{\alpha{}N}$ for some $N\geq N_0$.

By Lemma~\ref{lem:intervals}, for any $N_0 \leq n \leq N$ and any interval $\lbrac{y,y+\frac{1}{n}} \in \AI^n$, we have that
either $\norm{y-\frac{b}{q}} < \frac{W+1}{n}$ for some Farey point $\frac{b}{q}$ in $\FP_1^{W-1}$,
or $\lbrac{y,y+\frac{1}{n}}$ contains one of the intervals from $\FI_{\ceil{\beta{}n}}^{\floor{\alpha{}n}}$.
As the first $f(\floor{\alpha{}n})$ elements of $X$ is an $f$-piercing sequence of order $\floor{\alpha{}n}$, we have that for each interval $I$ in $\FI_{1}^{\floor{\alpha{}n}}$,
there is at least one $x_i \in I$ for some $i=1,2,\ldots{},f(\floor{\alpha{}n})$.

Our goal is to construct a strongly $g$-piercing sequence $Z$ of order $N$.
The sequence $Z$ will include $X$ as a subsequence, and already thereby will have a point in large majority of intervals of length $\frac{1}{n}$ from $[0,1)$. Intuitively, it remains to appropriately fill in the constructed $Z$ with relatively small number of points which will handle intervals $\lbrac{y,y+\frac{1}{n}}$ with $n < N_0$ and those with $n\geq N_0$ and $y$ close to some Farey point in $\FP_1^{W-1}$.

Now, for any nonnegative integer $r$ we define the set
$$H_r^W = \set{ \frac{b}{q} + \frac{a}{2^{r}} : \frac{b}{q} \in \FP_1^{W-1}, -2(W+1) \leq a \leq 2(W+1)} \cap \lbrac{0,1}\textrm{.}$$
As there are less than $(W-1)^2$ points in $\FP_1^{W-1}$, there are less than $5W^3$ elements in any $H_r^W$.
Observe, that for every $2^{r-1} \leq n \leq 2^{r}$, every $\frac{b}{q} \in \FP_1^{W-1}$, and every $y$ with $\norm{y-\frac{b}{q}} < \frac{W+1}{n}$, the interval $\lbrac{y,y+\frac{1}{n}}$ contains at least one of the points in $H_r^W$.
In what follows, we treat $H_r^W$ as a sequence of points arranged in any order.

We are ready to construct a sequence $Z$ that combines sequences $H_r^W$ with $X$.
We construct $Z$ incrementally by adding some sequences to the end. For two finite sequences $P$ and $Q$, we write $P \odot Q$ to denote their concatenation.

First, we explicitly add points piercing
all intervals in $\AI^1,\AI^2,\ldots{},\AI^{N_0}$, and elements $x_1,\ldots,x_{f(1)}$.
Let $$Z_0 = \seq{ \frac{1}{N_0}, \frac{2}{N_0}, \ldots,\frac{N_0-1}{N_0}} \odot \seq{x_1,\ldots,x_{f(1)}}\textrm{.}$$
Next, let $l=\floor{\log_2 \floor{\alpha{}N}}$, and for each $r=1,\ldots,l$, we extend $Z$ with points from $H_r^W$ and part of $X$.
Let $$Z_r = Z_{r-1} \odot H_r^W \odot \seq{ x_{f(2^{r-1})+1}, \ldots, x_{f(2^r)} }\textrm{.}$$
Finally, let $$Z = Z_{l} \odot H_{l+1}^W \odot \seq{ x_{f(2^{l})+1}, \ldots, x_{f(\floor{\alpha{}N})}}\textrm{.}$$

Now we show that $Z$ is strongly $g$-piercing of order $N$.
For every $1 \leq n \leq N_0$ we have $g(n) \geq N_0$, and every interval of length $\frac{1}{n}$ within $[0,1)$ contains one of the points from $Z_0$.
For every $N_0 < n \leq N$, the first $g(n)$ elements of $Z$ contain:
\begin{itemize}
\item $\set{\frac{i}{N_0}:1 \leq i < N_0}$ (first $N_0$ elements),
\item $H_1^W\cup H_2^W\cup\ldots\cup H_{\ceil{\log n}}^W$ (fewer than $5W^3\ceil{\log_2 n}$ elements),
\item $\seq{x_1,x_2,\ldots,x_{f(\floor{\alpha{}n})}}$ (as $f(\floor{\alpha{}n}) < \alpha^2 \gamma n$),
\end{itemize}
and hence, for each interval $I = \lbrac{y,y+\frac{1}{n}}$ within $[0,1)$, if $\norm{y-\frac{b}{q}} < \frac{W+1}{n}$ for some $\frac{b}{q} \in \FP_1^{W-1}$, then $I$ contains one of the points in $H_{\ceil{\log n}}^W$.
Otherwise, $I$ contains one of the Farey intervals in $\FI_{\ceil{\beta{}n}}^{\floor{\alpha{}n}}$, which in turn includes one of the elements in $\seq{x_1,x_2,\ldots,x_{f(\floor{\alpha{}n})}}$.

This ends the proof.
\end{proof}

\section{Proofs of the main results}

We are ready to prove our two main theorems stated in the Introduction.
\thmgood*
\begin{proof}
The \emph{if} part follows immediately from Corollary~\ref{Corollary dBE}.

For the \emph{only if} part, suppose that there exists some infinite $f$-piercing sequence $X$ with $f(n)=\gamma{}n+\oh{n}$ and $\gamma<\frac{1}{\ln 2}$.
Choose $W \geqslant 2$ large enough, so that for $\alpha = \frac{W+1}{W-1}$ we have $\alpha^2\gamma < \frac{1}{\ln 2}$.
Let $N_0$ and $g(n) = \alpha^2\gamma n + \oh{n}$ be the constant and the functions guaranteed by Lemma~\ref{lem:goodgreat}.

For every $N>N_0$ we can apply Lemma~\ref{lem:goodgreat} to the prefix of the first $\floor{\alpha{}N}$ elements of $X$ to obtain a sequence $Z_N$.
It is obvious from the construction used in Lemma~\ref{lem:goodgreat} that the sequence of sequences $Z_N$ for all $N>N_0$ converges to an infinite sequence $Z$, and that the sequence $Z$ is strongly $g$-piercing.
This contradicts Theorem~\ref{Theorem dBE}.
\end{proof}

\thms*

\begin{proof}
For the lower bound, we may use the strongly $(\ceil{\frac{n}{\ln 2}})$-piercing sequence $X$ given by de~Bruijn and Erd\H{o}s, defined by $x_i = \log_2(2i+1)~(\modulo 1)$, for $i=1,2,\ldots,\ceil{\frac{N}{\ln 2}}$.
For $N\leq \floor{\frac{\ln 2}{1-\ln 2}d}$, we have that
$N+d \geqslant \ceil{\frac{N}{\ln 2}}$ and $X$ is strongly $(n+d)$-piercing of order $N$.

For the upper bound, assume to the contrary that $s(d)$ is not $\frac{1+\ln2}{1-\ln2}d+\oh{d}$.
Then there exists a constant $c>\frac{1+\ln2}{1-\ln2}$ such that $s(d) > cd$ for infinitely many values of $d$.
Choose an integer $W>10$ large enough so that for $\alpha{}=\frac{W+1}{W-1}$ we have $\frac{1+\alpha^2\ln2}{1-\alpha^2\ln 2}<c$.

Then, choose $d$ large enough so that

$$s(d) > \ceil{\frac{1+\alpha^2\ln2}{1-\alpha^2\ln 2} \cdot d + 500W^3},$$
and for $N_0=\floor{\frac{s(d)}{\alpha}}$ we additionaly have
$$\frac{1}{H_{2\floor{\frac{N_0}{2}}}-H_{\floor{\frac{N_0}{2}}-1}} > \frac{1}{\alpha{}\ln 2}.$$
Observe that
$$N_0=\floor{\frac{s(d)}{\alpha{}}}>\frac{1+\alpha^2\ln2}{\alpha{}(1-\alpha^2\ln2)}d+400W^3,$$ and let $X=(x_1,x_2,\ldots,x_{\floor{\alpha{}N_0}+d})$ be an $(n+d)$-piercing sequence of order $\floor{\alpha{}N_0}$.

Let $$Z = H_{\ceil{\log_2 \frac{N_0}{8}}}^W \odot H_{\ceil{\log_2 \frac{N_0}{4}}}^W \odot H_{\ceil{\log_2 \frac{N_0}{2}}}^W \odot H_{\ceil{\log_2 N_0}}^W \odot X$$ be a sequence of length $2L$ (as it will be clear that we can remove its last element if necessary).
We have $2L>\alpha{}N_0$ and $L>\frac{N_0}{2}$.
Observe that for every $L \leq l \leq 2L$, we have that the first $l$ elements of $Z$ includes sequences $H_r^W$, and at least first $l-20W^3-1$ elements of $X$.
Observe that for $M=l-20W^3-1-d$, as $X$ is $(n+d)$-piercing, we have that the first $l-20W^3-1$ elements of $X$
include a point in every interval $[\frac{i}{m},\frac{i+1}{m}]$ for every $m=1,2,\ldots,M$, and $i=0,1,\ldots,m-1$.
Further, $\frac{N_0}{8} < \frac{M}{\alpha}  \leq {N_0}$, and thus Lemma~\ref{lem:intervals} implies that the maximum distance between consecutive first $l$ elements of $Z$ is at most $\frac{\alpha}{M}$.

On the other hand, applying Lemma~\ref{lem:gamma2n} to $Z$ we get that for some $L \leq l \leq 2L$ we have that the maximum distance between consecutive elements of $\set{0,1,Z_1,Z_2,\ldots,Z_l}$ equals at least
$$\frac{1}{(H_{2L}-H_{L-1})l} > \frac{1}{(H_{2\floor{\frac{N_0}{2}}}-H_{\floor{\frac{N_0}{2}}-1})l} > \frac{1}{\alpha{}l\ln2}.$$
It remains to prove that $\frac{\alpha}{M} \leq \frac{1}{\alpha{}l\ln2}$ to get a contradiction.

Starting with the assumption on the order of $s(d)$ we get the following:
\begin{align*}
s(d) & \geq \frac{d(1 + \alpha^2\ln2)+2+40W^3}{1-\alpha^2\ln2} & \textrm{/$+d$} \\
2L \geq s(d) + d & \geq \frac{2d+2+40W^3}{1-\alpha^2\ln2} & \textrm{/$\div 2$} \\
l \geq L & \geq \frac{d+1+20W^3}{1-\alpha^2\ln2}\\
l(1-\alpha^2\ln2) & \geq d+1+20W^3\\
M=l-20W^3-1-d &\geq \alpha^2l\ln2 & \textrm{/$\div\alpha$}\\
\frac{M}{\alpha} & \geq \alpha{}l\ln2.\\
\end{align*}
\end{proof}

\section{Final remarks}

Let us conclude the paper with some discussion on possible future directions of studies in this topic.

Perhaps the most natural and challenging one concerns the determination of the concrete values of the function $s(d)$. We know only that $s(0)=17$, by the result of Warmus~\cite{Warmus76}, and that $s(1)\geq 31$, by an explicit example found by Oliveira e Silva~\cite{Oliveira17}, who also claims that $s(1)=31$ was verified by computer search.
Some computational experiments were also made by Levy in his PhD thesis~\cite{LevyPhD}, but no conjecture is stated there. Even now, though we know of quite close linear bounds from both sides for the function $s(d)$, a more exact formula for $s(d)$ seems rather elusive. However, it would be interesting to determine exactly the multiplicative constant in the presumed asymptotics of $s(d)$. We conjecture that our lower bound is the correct value.

\begin{conjecture}The limit $\lim_{d\to \infty}\frac{s(d)}{d} = \frac{\ln2}{1-\ln2}$.
\end{conjecture}

Let us stress, however, that it is not known whether the above limit actually exists.

Going back to the original Steinhaus' problem, one may ask for possible multidimensional versions. Here is one of them for sequences of points in the plane. Imitating the original piercing property, we say that a sequence of points $X$ in the unit square is \emph{piercing} if for every finite prefix of $X$ of length $n$ there exists a tiling of the square into $n$ rectangles, each of area $\frac{1}{n}$, such that every rectangle contains exactly one point of the prefix. Is it true that there exist arbitrarily long piercing sequences in the unit square?


\begin{thebibliography}{10}

\bibitem{BerlekampG70}
Elwyn~R. Berlekamp and Ronald~L. Graham.
\newblock Irregularities in the distributions of finite sequences.
\newblock {\em Journal of Number Theory}, 2(2):152--161, 1970.
\newblock \href {https://doi.org/10.1016/0022-314X(70)90015-6}
  {\path{doi:10.1016/0022-314X(70)90015-6}}.

\bibitem{BruijnE49}
Nicolaas~G. de~Bruijn and Paul Erd{\"o}s.
\newblock Sequences of points on a circle.
\newblock {\em Indagationes Mathematicae}, 11:14--17, 1949.

\bibitem{Graham13}
Ronald~L. Graham.
\newblock A note on irregularities of distribution.
\newblock {\em Integers}, 13:\#A53:1--4, 2013.
\newblock URL: \url{http://math.colgate.edu/~integers/n53/n53.pdf}.

\bibitem{Konyagin21}
Sergei Konyagin.
\newblock On irregularity of finite sequences.
\newblock {\em Proceedings of the Steklov Institute of Mathematics},
  314:90--95, 2021.
\newblock \href {https://doi.org/10.1134/S0081543821040052}
  {\path{doi:10.1134/S0081543821040052}}.

\bibitem{LevyPhD}
Karl Levy.
\newblock {\em Some Results in Combinatorial Number Theory}.
\newblock PhD thesis, The City University of New York, 2017.

\bibitem{Levy20}
Karl Levy.
\newblock Lower and upper bounds on irregularities of distribution.
\newblock {\em Integers}, 20:\#A26:1--17, 2020.
\newblock URL: \url{http://math.colgate.edu/~integers/u26/u26.pdf}.

\bibitem{NivenZ72}
Ivan~M. Niven and Herbert~S. Zuckerman.
\newblock {\em An Introduction to the Theory of Numbers}.
\newblock John Wiley and Sons., 1972.

\bibitem{Oliveira17}
Tom{\'a}s Oliveira~e Silva.
\newblock A problem related to the 17 point problem of {S}teinhaus.
\newblock (version: 2017-01-20).
\newblock URL: \url{https://mathoverflow.net/q/260116}.

\bibitem{Ramshaw78}
Lyle Ramshaw.
\newblock On the gap structure of sequences of points on a circle.
\newblock {\em Indagationes Mathematicae (Proceedings)}, 81(1):527--541, 1978.
\newblock \href {https://doi.org/10.1016/S1385-7258(78)80043-2}
  {\path{doi:10.1016/S1385-7258(78)80043-2}}.

\bibitem{Roth54}
Klaus~F. Roth.
\newblock On irregularities of distribution.
\newblock {\em Mathematika}, 1(2):73--79, 1954.
\newblock \href {https://doi.org/10.1112/S0025579300000541}
  {\path{doi:10.1112/S0025579300000541}}.

\bibitem{Sos}
Vera S\'{o}s.
\newblock On the distribution mod 1 of the sequence n{$\alpha$}.
\newblock {\em Annales Universitatis Scientiarium Budapestinensis de Rolando
  E\"{o}tv\"{o}s Nominatae Sectio Mathematica}, 1:127--134, 1958.

\bibitem{SteinhausFD}
Hugo Steinhaus.
\newblock The problem of fair division.
\newblock {\em Econometrica}, 16:101--104, 1948.

\bibitem{SteinhausZM}
Hugo Steinhaus.
\newblock Liczby z{\l}ote i \.zelazne.
\newblock {\em Applicationes Mathematicae}, 3:51--65, 1958.
\newblock \href {https://doi.org/10.4064/am-3-1-51-65}
  {\path{doi:10.4064/am-3-1-51-65}}.

\bibitem{Steinhaus58}
Hugo Steinhaus.
\newblock {\em Sto zada\'n}.
\newblock PWN Warszawa, 1958.

\bibitem{Steinhaus64}
Hugo Steinhaus.
\newblock {\em One Hundred Problems in Elementary Mathematics}.
\newblock Basic Books, 1964.

\bibitem{Warmus76}
Mieczys{\l}aw Warmus.
\newblock A supplementary note on the irregularities of distributions.
\newblock {\em Journal of Number Theory}, 8(3):260--263, 1976.
\newblock \href {https://doi.org/10.1016/0022-314X(76)90002-0}
  {\path{doi:10.1016/0022-314X(76)90002-0}}.

\bibitem{Swierczkowski}
Stanisław \'Swierczkowski.
\newblock On successive settings of an arc on the circumference of a circle.
\newblock {\em Fundamenta Mathematicae}, 46:187--189, 1959.

\end{thebibliography}
\end{document}